\documentclass[11pt]{amsart}

\usepackage{amsmath,amssymb,amsthm}

\usepackage[T1]{fontenc}
\usepackage[utf8]{inputenc}
\usepackage{lmodern}
\usepackage{microtype}
\usepackage{enumitem}
\usepackage[hidelinks]{hyperref}
\usepackage[noabbrev,capitalize]{cleveref}
\usepackage{fullpage}

\newtheorem{theorem}{Theorem}[section]

\newtheorem{lemma}[theorem]{Lemma}
\newtheorem{corollary}[theorem]{Corollary}
\theoremstyle{remark}

\newtheorem{conjecture}[theorem]{Conjecture}

\newcommand{\E}{\mathbb E}
\newcommand{\Prob}{\mathbb P}
\newcommand{\F}{\mathcal F}
\newcommand{\1}{\mathbf 1}

\title{Almost sure upper bound for sums of random multiplicative functions and critical chaos}
\author{William Verreault}
\date{}

\address{Department of Mathematics \\
University of Toronto   \\ 
Toronto, ON\\
Canada}
\email{william.verreault@utoronto.ca}  

\begin{document}

\vspace*{-4mm}

\begin{abstract}
Let $f$ be a Steinhaus or Rademacher random multiplicative function. We use methods from the theory of critical chaos to improve on the best known upper bound for partial sums of random multiplicative functions. In particular, our results imply that for any $\varepsilon>0$, almost surely
$$
\Big|\sum_{n\le x}f(n)\Big|
 \ll_{\varepsilon}\sqrt{x}(\log_2x)^{1/4}(\log_3x)^{1+\varepsilon}.
$$
This proves in a strong form a conjecture of Harper on large fluctuations of partial sums of random multiplicative functions, and determines the exact corresponding logarithmic exponent.
\end{abstract}

\maketitle

\section{Introduction}

\subsection{Background}
We consider the two standard models of random multiplicative functions. In the \emph{Steinhaus} model, let $(f(p))_p$ be independent random variables indexed by primes and uniformly distributed on the unit circle. We set $f(1)=1$ and extend $f$ completely multiplicatively to $\mathbb N$. In the \emph{Rademacher} model, let $(f(p))_p$ instead be independent random variables uniformly distributed on $\{\pm1\}$. We set $f(1)=1$, extend $f$ multiplicatively to the squarefree integers, and set $f(n)=0$ when $n$ is not squarefree. This model was introduced by Wintner \cite{Wintner} as a probabilistic analogue of the M\"obius function in number theory, while Steinhaus random multiplicative functions serve as probabilistic models for Dirichlet characters $n\mapsto \chi(n)$ and for the Archimedean characters $n\mapsto n^{it}$. 

The study of these functions, especially through their partial sums
\begin{equation*}
  M_f(x)=\sum_{n\le x}f(n),
\end{equation*} has attracted considerable attention at the interface
of number theory, probability, and analysis in recent years, e.g., \cite{Hough}, \cite{ChatterjeeSoundararajan}, \cite{HarperLimit}, \cite{LTW}, \cite{HNR}, \cite{GranvilleHarperSoundararajan}, \cite{BondarenkoSeip}, \cite{HeapLindqvist}, \cite{HarperHigh}, \cite{HarperLow}, \cite{AggarwalChaos}, \cite{AggarwalEtAl}, \cite{Mastrostefano}, \cite{SoundararajanZaman}, \cite{HarperLarge}, \cite{KlurmanShkredovXu},  \cite{SoundararajanXu},  \cite{Caich}, \cite{CaichShortIntervals},  \cite{GorodetskyWongCLT},  \cite{HardyWeighted},  \cite{HardyExponential},  \cite{PandeyWangXu},  \cite{XuCancellation}, \cite{ChinisShala}, \cite{GorodetskyWongLimit}, \cite{HardyLargePrime}, \cite{HofmannEtAl}, \cite{HobanEtAl}, \cite{HarperSoundararajanXu}, and \cite{ShalaLarge}. 
These works concern, among other topics, moments and limiting
distributions, almost sure fluctuations, short intervals and
restricted sums, and analogous models over
function fields. The list is far from exhaustive; for broader background and further
references, we refer to Harper's recent survey \cite{HarperSurvey} and to
the extensive introductions of \cite{HarperLow,GorodetskyWongLimit}.

In this paper, we are concerned with the almost sure fluctuations of $M_f(x)$.

\subsection{Main result}
The goal of this paper is to prove the following improvement to the known almost sure upper bounds for partial sums of random multiplicative functions. Here $\log_k x$ denotes the $k$-th iterate of the natural logarithm.

\begin{theorem}\label{thm:main}
Let $G:[3,\infty)\to(0,\infty)$ be eventually nondecreasing and suppose that
\begin{equation*}
  \sum_{\ell\ge3}\frac{\ell\log\ell}{G(\ell)^2}<\infty.
\end{equation*}
For a Steinhaus or Rademacher random multiplicative function $f$, almost surely
as $x\to\infty$,
\begin{equation*}
  |M_f(x)|
  \ll_{G}
  \sqrt{x}(\log_2x)^{1/4}G(\log_3x).
\end{equation*}
\end{theorem}

The implied constant in this theorem
can be chosen independently of the realization of $f$. Only the threshold beyond which the estimates hold may depend on the realization.
Because $G$ is eventually nondecreasing, comparison on each interval $[\ell,\ell+1]$ shows that the summability hypothesis is equivalent, after modifying $G$ on a bounded interval, to
\begin{equation*}
  \int_3^\infty\frac{t\log t}{G(t)^2}\,dt<\infty.
\end{equation*}

\subsection*{Remark}
This work was initiated following an Aisenstadt Chair Lecture by Adam Harper during the workshop
\emph{Probability in Number Theory} at the Centre de recherches
math\'ematiques.  While this manuscript was being completed, Durkan and
Pearce-Crump independently proved Harper's conjecture
\cite{DurkanPearceCrump}.  Their argument is different from the one given
here, and the present method yields the more precise family of bounds in
\cref{thm:main}.  We compare the two approaches in \cref{sec:proof}.

\subsection{History of the problem}
The first almost sure results go back to Wintner \cite{Wintner}, who proved in the Rademacher model that, for every $\varepsilon>0$,
\begin{equation*}
M_f(x)\ll_{\varepsilon}x^{1/2+\varepsilon},
\end{equation*}
while an estimate $M_f(x)=O(x^{1/2-\varepsilon})$ fails almost surely.  Hal\'asz \cite{Halasz} later obtained the much sharper upper bound
\begin{equation*}
M_f(x)\ll\sqrt{x}\exp\big(C\sqrt{\log_2x\log_3x}\big)
\end{equation*}
in the same model. Subsequently, Basquin \cite{Basquin} and Lau--Tenenbaum--Wu \cite{LTW} independently proved that for every $\varepsilon>0$,
\begin{equation*}
  M_f(x)\ll_{\varepsilon}\sqrt{x}(\log_2x)^{2+\varepsilon}
\end{equation*}
almost surely in the Rademacher model. These results still left open whether $M_f(x)=O(\sqrt{x})$ might hold almost surely.  Hal\'asz raised this question in the Steinhaus model, while Erd\H{o}s conjectured in the Rademacher model that such a bound should fail.

A decisive change in the picture came from Harper's determination of all low moments.  Uniformly for $0\leq q\leq1$ and in both models, he proved \cite{HarperLow} that
\begin{equation*}
  \E|M_f(x)|^{2q}
  \asymp
  \Big(\frac{x}{1+(1-q)\sqrt{\log_2x}}\Big)^q.
\end{equation*}
For $q=1$, this agrees with a trivial calculation using orthogonality, but taking $q=1/2$ gives the striking estimate
\begin{equation*}
  \E|M_f(x)|
  \asymp\frac{\sqrt{x}}{(\log_2x)^{1/4}}.
\end{equation*}
Thus the first moment exhibits better than square-root cancellation by a factor $(\log_2x)^{1/4}$, confirming a conjecture of Helson and contradicting the behaviour suggested by the second moment alone.  Harper's proof connected this separation of the low moments to critical multiplicative chaos, a connection that also plays a central role here, as described in \cref{sec:chaos}.

The natural probabilistic benchmark for this problem is the simple random walk. If $(X_n)_n$ is a sequence of independent ``coin flips'' taking the values $+1$ and $-1$ with equal probability, and $S_N=\sum_{n\leq N}X_n$, then $S_N$ has size of order $\sqrt N$ at a fixed time, whereas the law of the iterated logarithm gives
\begin{equation*}
 \limsup_{N\to\infty} \frac{|S_N|}{\sqrt{2N\log_2N}}=1
\end{equation*}
almost surely. In other words, taking the largest fluctuations over time introduces a factor of order $(\log_2N)^{1/2}$ beyond the fixed-time scale.
Harper observed that this gives a compelling heuristic for the almost sure problem:  
if the same amplification is applied to the fixed-$x$ scale suggested by the low moments of $M_f(x)$, one obtains
\begin{equation*}
  \frac{\sqrt{x}}{(\log_2x)^{1/4}}(\log_2x)^{1/2}
  =\sqrt{x}(\log_2x)^{1/4}.
\end{equation*}
The partial sums of a random multiplicative function have neither independent values nor independent increments, but this nevertheless identifies the exponent $1/4$ as the natural candidate for their largest almost sure fluctuations. Harper emphasized, however, that proving such an upper bound would remain ``a formidable task'' \cite{HarperLarge}. We shall discuss the law of the iterated logarithm heuristic further in light of our results in \cref{sec:chaos}.

Harper established the corresponding lower bound \cite{HarperLarge} by showing that for every function $V(x)\to\infty$, almost surely there are arbitrarily large $x$ for which
\begin{equation*}
  |M_f(x)|\ge
  \frac{\sqrt{x}(\log_2x)^{1/4}}{V(x)}.
\end{equation*}
Thus Harper answered Halász’s question negatively and confirmed Erdős’s prediction in the Rademacher model.

Motivated by this lower bound and the above heuristic, the assertion that, in either model and for every $\varepsilon>0$, one almost surely has
\begin{equation*}
  M_f(x)\ll_{\varepsilon}
  \sqrt{x}(\log_2x)^{1/4+\varepsilon}
\end{equation*}
has become known as Harper's conjecture.  Mastrostefano \cite{Mastrostefano} proved the conjectured estimate for the contribution from integers having a prime factor larger than $\sqrt{x}$.  This is the part of the sum most directly amenable to conditioning on the smaller primes, since the largest prime factor then occurs only once.  Caich \cite{Caich} was the first to treat the full sum through a martingale indexed by the largest prime factor and obtained, in both models,
\begin{equation*}
  M_f(x)\ll_{\varepsilon}
  \sqrt{x}(\log_2x)^{3/4+\varepsilon}
\end{equation*}
almost surely. This reduced the problem to closing the remaining gap in the exponent of $\log_2x$.

\subsection{Corollaries of the main result}
The theorem allows for a range of upper bounds rather than a single fixed loss.

\begin{corollary}
In either model, for every $\varepsilon>0$, almost surely
\begin{equation*}
  |M_f(x)|
  \ll_{\varepsilon}
  \sqrt{x}(\log_2x)^{1/4}
  \log_3x\,\log_4x\,(\log_5x)^{1/2+\varepsilon}.
\end{equation*}
Consequently,
\begin{equation*}
  |M_f(x)|
  \le \sqrt{x}(\log_2x)^{1/4}(\log_3x)^{1+o(1)}
\end{equation*}
almost surely as $x\to\infty$.  In particular, Harper's conjecture holds.
\end{corollary}

\begin{proof}
Let $\varepsilon>0$. For all sufficiently large $t$, take
\begin{equation*}
  G(t)=t\log t\,(\log_2t)^{1/2+\varepsilon},
\end{equation*}
and modify $G$ on a bounded interval so that it is nondecreasing. Then
\begin{equation*}
  \frac{\ell\log\ell}{G(\ell)^2}
  =\frac1{\ell\log\ell\,(\log_2\ell)^{1+2\varepsilon}},
\end{equation*}
whose sum converges by the integral test.  
\end{proof}

The displayed choice of $G$ is only a convenient one. By iterating the integral test, for any fixed integer $r\geq 2$, one may instead take
$$
G(t)=t\log t\prod_{j=2}^{r-1}(\log_j t)^{1/2}(\log_r t)^{1/2+\varepsilon}
$$
for all sufficiently large $t$. Thus one may obtain corresponding refinements involving any prescribed finite number of iterated logarithms.

Combining this upper bound with Harper's lower bound, applied with $V(x)=(\log_2x)^{\delta}$ and then with $\delta\to 0$ along a countable sequence, yields the following exact limsup statement for the exponent of $\log_2x$. It also appears in \cite{DurkanPearceCrump}.

\begin{corollary}
For either model, almost surely
\begin{equation*}
  \limsup_{x\to\infty}
  \frac{\log(1+|M_f(x)|/\sqrt{x})}{\log_3x}
  =\frac14.
\end{equation*}
\end{corollary}


\subsection{Proof strategy and comparison with earlier work} \label{sec:proof}
The proof combines three principal ingredients from earlier work.  We use the interpolation lemma of Lau--Tenenbaum--Wu \cite{LTW} to reduce the problem to a sparse collection of test points in macroblocks. Following the largest prime factor approach developed by Caich \cite{Caich}, we then isolate a main term that is a martingale as the primes are revealed.
After a short averaging in the prime variable, its quadratic variation (or clock) is controlled by consecutive sums of normalized $L^2(dz/z^2)$-energies of smooth partial sums.
Standard estimates for smooth numbers then show that only a consecutive window of $O(\log_3x)$ prime layers can contribute substantially.

To control this window, we use Harper's critical chaos estimates from \cite{HarperLow} through an exact identity supplied by Parseval's formula for Dirichlet series. The resulting normalized whole-line energy is a nonnegative mean-one martingale in the prime cutoff.  We stop this energy at a single threshold throughout each macroblock. An $\ell^2$-valued Doob inequality then turns the stopped terminal energy bound into conditional exponential moments for the maximal energy of each layer.  These estimates may be iterated across consecutive layers without assuming any independence between them. The main new point is to control the relevant window as a whole, rather than estimating its layers separately and then accumulating the resulting losses. This organization also permits coarser prime layers of bounded reciprocal-prime mass and considerably simplifies the central estimate for the quadratic variation.

As mentioned above, Durkan and Pearce-Crump have independently proved Harper's conjecture \cite{DurkanPearceCrump}.
Their argument also combines stopping with higher-moment estimates, but
retains Caich's finer prime-block decomposition and uses a deterministically
tilted supermartingale together with scalar maximal inequalities,
hypercontractivity, and divisor weights.  Here the stopped object is instead
an exactly normalized Parseval energy, and the vector-valued maximal
estimate feeds directly into the consecutive window bound.  This both
simplifies the central clock estimate and yields the full family of bounds
in \cref{thm:main}.  We do not know whether their method can be adapted to
recover our stronger conclusion.

\subsection*{Remark}
Some of the techniques used here are similar to those in a recent joint work of Orsoni and the author \cite{OrsoniVerreault} on Fourier decay of
Gaussian multiplicative chaos. This is perhaps unsurprising, since partial sums of random multiplicative functions are expected to behave in a similar way as the Fourier coefficients of Gaussian multiplicative chaos \cite{GarbanVargas}, and similarly for other random models such as holomorphic multiplicative chaos \cite{NajnudelPaquetteSimm}.
A concrete connection already appears in Harper's lower-bound argument from \cite{HarperLarge}. After conditioning on the small primes, covariances of rescaled partial sums at endpoints $x$ and $y$ are related to oscillatory Euler product integrals with phase $e^{it\log(x/y)}$. These can be viewed, schematically, as Fourier transforms of approximate critical multiplicative chaos at frequency $\log(x/y)$.

\subsection{Critical chaos and a conjectural law of the iterated logarithm} \label{sec:chaos}

Critical chaos enters the proof naturally through the quadratic variation of the largest prime martingale. More precisely, we 
compare it with the total mass of a normalized density built from the squared modulus of the truncated Euler product $F_y$.
The reason this Euler product mass is critical can be seen most directly in the Steinhaus model.  The leading part of $\log|F_y(1/2+it)|$ is a one-dimensional log-correlated field, and the normalized squared modulus of the Euler product has the
structure of a critical, approximately Gaussian multiplicative chaos (see, e.g., \cite{RhodesVargas,Berestycki} for more on Gaussian multiplicative chaos). This critical chaos picture, already identified by Harper in \cite{HarperLow}, has proved fruitful in many recent works on random multiplicative functions.
The gain from this point of view comes from the critical chaos phenomenon that 
the normalized mass is generally much smaller than its mean, with that mean sustained by rare large values. A key ingredient in our refinement is that we make the critical chaos mechanism explicit through an exactly mean-normalized energy and exploit Harper's fractional moment estimates uniformly as the moment order approaches $1$ from below.

While the exponent $1/4$ of $\log_2x$ is sharp, the precise almost sure order of $M_f(x)$ remains open. In the Steinhaus model,
Gorodetsky and Wong \cite{GorodetskyWongLimit}
proved that
$$
 \frac{(\log_2x)^{1/4}}{\sqrt{x}}M_f(x)
 \xrightarrow{d}\sqrt{V_\infty}\,\mathcal G,
$$
where $V_\infty$ is a positive critical chaos mass constructed
from the same random multiplicative function $f$, and $\mathcal G$ is an independent standard complex Gaussian with $\E|\mathcal G|^2=1$. More concretely, in the notation of \cref{lem:critical-energy}, Theorem 1.1 and Corollary 1.2 of 
\cite{GorodetskyWongLimit} give
$$
 \sqrt{\log_2y}\,Z_y
 \xrightarrow[y\to\infty]{\Prob}V_\infty.
$$
Thus $V_\infty$ is the limit in probability of the critically
rescaled energy appearing in our proof. Combining this distributional picture with Harper's  \cite{HarperLarge} heuristic
for large fluctuations leads us to the
following conjectural law of the iterated logarithm, whose
leading constant is determined by the critical chaos mass.

\begin{conjecture}
Let $f$ be a Steinhaus random multiplicative function, and let
$V_\infty$ be the associated random variable in
\cite[Theorem~1.3]{GorodetskyWongLimit}. Almost surely,
$$
 \limsup_{x\to\infty}
 \frac{|M_f(x)|}{\sqrt{x}(\log_2x)^{1/4}}
 =\sqrt{V_\infty}.
$$
\end{conjecture}

\section{Preliminaries}

\subsection{Notation}
Let $f:\mathbb N \to \mathbb C$ be either a Steinhaus or Rademacher random multiplicative function.
We write $P(n)$ for the largest prime factor of $n$, with $P(1)=1$.  For
$u,y>0$, define the following partial sums of $f$ restricted to $y$-smooth integers:
\begin{equation*}
  \Psi_f(u,y)=\sum_{\substack{n\le u\\P(n)\le y}}f(n),
  \qquad
  \Psi_f'(u,y)=\sum_{\substack{n\le u\\P(n)<y}}f(n).
\end{equation*}
Both sums are zero when $u<1$.  We use the prime filtrations
\begin{equation*}
  \F_y=\sigma(f(p):p\le y),
  \qquad
  \F_{p^-}=\sigma(f(q):q<p).
\end{equation*}
Note that $\F_y$ includes the prime at
the endpoint when $y$ itself is prime.  Between consecutive primes, the
filtration and all processes indexed by a prime cutoff are taken to
be constant.

For variables $a$ and $b$, we write $a \ll b$ or $a = O(b)$ to say that there exists an absolute positive constant $C$ such that $|a| \leq C|b|$. If the constant $C$ depends on a parameter, say $k$, we shall write $a \ll_k b$ or $a = O_k(b)$. We also write $a\asymp b$ if both $a\ll b$ and $b\ll a$ hold. If $a$ and $b$ depend on a positive parameter $x$, then we say that $a=o(b)$ (as $x \rightarrow \infty$) if the ratio $a/b$ converges to $0$ as $x \rightarrow \infty$. Dependence on the choice between the Steinhaus and Rademacher
models is suppressed throughout.

Indicator functions for a set $A$ will always be indicated by $\1_A$, while its cardinality will be denoted either $\#A$ or $|A|$. The variable $p$ will be reserved for prime numbers. Finally, we write $\exp(x)$ instead of $e^x$ as soon as the exponential has more than one level or contains complicated expressions.


\subsection{Probabilistic tools}
We refer to any textbook on graduate probability such as \cite{Durrett} for the elementary facts and inequalities used below.

We first record two martingale estimates in the forms needed below. The first is a direct specialization of Caich's stopped Hoeffding estimate.

\begin{lemma}\label[lemma]{lem:subgaussian}
Let $(\mathcal G_k)_{0\le k\le N}$ be a filtration.  Suppose that $\xi_k$ is
$\mathcal G_k$-measurable and, conditionally on $\mathcal G_{k-1}$, is either
a symmetric sign or a uniform point on the unit circle, and that $a_k$ is a
complex-valued $\mathcal G_{k-1}$-measurable random variable.  Put
\begin{equation*}
  S_N=\sum_{k=1}^N\xi_ka_k,
  \qquad
  V_N=\sum_{k=1}^N|a_k|^2.
\end{equation*}
Then, for every $t,v>0$,
\begin{equation*}
  \Prob(|S_N|>t,\ V_N\le v)\le2\exp\Big(\frac{-t^2}{10v}\Big).
\end{equation*}
\end{lemma}

\begin{proof}
Set
\begin{equation*}
  Q_0=0,\qquad
  \eta_k=\1_{\{Q_{k-1}+|a_k|^2\le v\}},\qquad
  Q_k=Q_{k-1}+\eta_k|a_k|^2.
\end{equation*}
The variables $\eta_k\xi_ka_k$ form a complex martingale difference sequence with predictable bounds $\eta_k|a_k|$, and $\sum_k\eta_k|a_k|^2\le v$. Thus \cite[Lemma~3.11]{Caich} gives
\begin{equation*}
\Prob\Big(\Big|\sum_{k=1}^N\eta_k\xi_ka_k\Big|>t\Big)
  \le2\exp(-t^2/(10v)).
\end{equation*}
On $\{V_N\le v\}$ one has $\eta_k=1$ for every $k$, which proves the claim.
\end{proof}

We shall also use the following $\ell^2$-valued extension of Doob's
maximal inequality.  

\begin{lemma}\label[lemma]{lem:vector-doob}
Let $(\mathcal G_k)_{0\le k\le N}$ be a filtration, let $r>1$, and let
$(X_m)_m$ be a finite family in $L^{2r}$.  For nonnegative weights $w_m$,
\begin{equation*}
 \Big\|
  \Big(\sum_mw_m\sup_{0\le k\le N}
   |\E[X_m\mid\mathcal G_k]|^2\Big)^{1/2}
 \Big\|_{L^{2r}}
 \ll \sqrt r\,
 \Big\|\Big(\sum_mw_m|X_m|^2\Big)^{1/2}\Big\|_{L^{2r}}.
\end{equation*}
The implied constant is absolute.
\end{lemma}

\begin{proof}
The general $\ell^q$-extension of Doob's inequality is standard and known in much more generality (see, e.g., \cite[Theorem~3.2.7]{HNVW}).  We include a short duality argument, available due to the Hilbert structure of this special case, in order to record the dependence on $r$.

Write $M_m^*=\sup_k|\E[X_m\mid\mathcal G_k]|$ and
$S=(\sum_mw_m|X_m|^2)^{1/2}$.  By duality in $L^r$, take a nonnegative
$h$ with $\|h\|_{L^{r'}}=1$, and put $h_N=\E[h\mid\mathcal G_N]$ and
\begin{equation*}
 h^*=\sup_{0\le k\le N}\E[h_N\mid\mathcal G_k]
     =\sup_{0\le k\le N}\E[h\mid\mathcal G_k].
\end{equation*}
Since $(M_m^*)^2$ is $\mathcal G_N$-measurable, the weighted Doob inequality at exponent $2$ \cite[Theorem~3.2.3]{HNVW}, followed by conditional Jensen, gives
\begin{equation*}
 \E h(M_m^*)^2
 =\E h_N(M_m^*)^2
 \le4\E h^*|\E[X_m\mid\mathcal G_N]|^2
 \le4\E h^*|X_m|^2.
\end{equation*}
Multiplying the preceding inequality by $w_m$, summing in $m$, applying H\"older's inequality, and then using scalar Doob in $L^{r'}$, we obtain
\begin{equation*}
 \E h\sum_mw_m(M_m^*)^2
 \le4\E h^*S^2
 \le4\|h^*\|_{L^{r'}}\|S^2\|_{L^r}
 \le4r\|S^2\|_{L^r}.
\end{equation*}
Taking square roots proves the result.  The same argument applies to complex variables because the absolute value of a complex martingale is a nonnegative submartingale.
\end{proof}

\subsection{Number theoretic tools}
For the classical estimates from number theory used below, we refer to \cite{MontgomeryVaughan}.

\begin{lemma}\label[lemma]{lem:arithmetic}
The following estimates hold.
\begin{enumerate}[label=\textnormal{(\roman*)}]
\item Uniformly for $y\ge3$,
\begin{equation*}
  \sum_{p\le y}\frac1p=\log_2y+O(1).
\end{equation*}
In particular,
\begin{equation*}
  \prod_{p\le y}(1+1/p)\asymp\log y,
  \qquad
  \prod_{p\le y}(1-1/p)^{-1}\asymp\log y.
\end{equation*}

\item If $b$ is sufficiently large and $2\le H\le b^{1/2}$, then
\begin{equation*}
  H\sum_{b/(1+1/H)<p\le b}\frac1p\ll\frac1{\log b}.
\end{equation*}

\item There is an absolute constant $C$ such that, for $z\ge1$ and $y\ge2$,
\begin{equation*}
  \#\{n\le z:P(n)\le y\}
  \ll z(\log y)^C
  \exp\big(-\frac{\log z}{2\log y}\big).
\end{equation*}
\end{enumerate}
\end{lemma}

\begin{proof}
Part~\textnormal{(i)} follows from Mertens' theorem.
For part~\textnormal{(ii)}, the interval has length
\begin{equation*}
  h=b-\frac{b}{1+1/H}=\frac{b}{H+1}.
\end{equation*}
Since $H\le b^{1/2}$, one has $h\gg b^{1/2}$ and $\log h\asymp\log b$.  The Brun--Titchmarsh theorem 
gives $O(b/(H\log b))$ primes in the interval, and $1/p\asymp1/b$ there.

For part~\textnormal{(iii)}, Rankin's argument
with $\sigma=1-1/(2\log y)$ gives
\begin{equation*}
 \#\{n\le z:P(n)\le y\}
 \le z^\sigma\prod_{p\le y}(1-p^{-\sigma})^{-1}.
\end{equation*}
Since $p^{-\sigma}\le e^{1/2}/p$ and $p^{-2\sigma}\le e/p^2$ for $p\le y$, Mertens' theorem and the convergent contribution of prime powers show that the product is $O((\log y)^C)$.  Also
$z^\sigma=z\exp(-\log z/(2\log y))$.
\end{proof}

\subsection{Reduction to test points by interpolation}

The following is Lau--Tenenbaum--Wu's interpolation lemma in the Rademacher
model, together with its Steinhaus extension.

\begin{lemma}\label[lemma]{lem:sparse-grid}
For every fixed $A>0$, there is $c_0=c_0(A)\in(0,1)$ such that, for
\begin{equation*}
  x_i=\lfloor\exp(i^{c_0})\rfloor,
\end{equation*}
one has almost surely
\begin{equation*}
  \max_{x_{i-1}<x\le x_i}
  |M_f(x)-M_f(x_{i-1})|
  \ll_{A}\frac{\sqrt{x_i}}{(\log x_i)^A}.
\end{equation*}
This holds in both models.
\end{lemma}

\begin{proof}
The Rademacher statement is \cite[Lemma~2.3]{LTW} (here one must use the
corrected version cited in the bibliography).  Its chaining proof uses
\begin{equation}\label{eq:LTWR}
  \E|M_f(v)-M_f(u)|^4
  \ll v^{2/3}(v-u)^{4/3}(\log v)^{52/3},
  \qquad v\ge u+1\ge 2.
\end{equation}
We verify a stronger estimate in the Steinhaus model.  For a finitely
supported sequence $(a_n)$, put
\begin{equation*}
  S=\sum_na_nf(n),
  \qquad
  b_m=\sum_{rs=m}a_ra_s.
\end{equation*}
Complete multiplicativity and orthogonality give
$
  \E|S|^4=\sum_m|b_m|^2.
$
Writing $d(n)$ for the divisor function, Cauchy--Schwarz and
$d(rs)\le d(r)d(s)$ yield
\begin{equation*}
  \sum_m|b_m|^2
  \le\sum_md(m)\sum_{rs=m}|a_r|^2|a_s|^2
  \le\Big(\sum_nd(n)|a_n|^2\Big)^2.
\end{equation*}
Take $a_n=\1_{\{u<n\le v\}}$ and use
H\"older's inequality to get
\begin{equation*}
\sum_{u<n\le v}d(n)
 \ll(v-u)^{2/3}
 \Big(\sum_{n\le v}d(n)^3\Big)^{1/3}.
\end{equation*}
Write $d_k(n)=\sum_{n_1\cdots n_k=n}1$ and note that $d(n)^3\leq d_8(n)$ for every $n$. Using a standard divisor-sum estimate of the form
\begin{equation*}
\sum_{n\leq x}d_k(n) \ll x(\log(2x))^{k-1}
\end{equation*}
for fixed $k$ gives
\begin{equation*}
  \E|M_f(v)-M_f(u)|^4
  \ll v^{2/3}(v-u)^{4/3}(\log(2v))^{14/3},
\end{equation*}
which is stronger than \eqref{eq:LTWR}. This is the only model-dependent input in the
argument of \cite[Lemma~2.3]{LTW}, so its subdivision and Borel--Cantelli steps apply unchanged.
\end{proof}

\subsection{Critical chaos and energy}

For $y\ge1$ and $\Re s>0$, let
\begin{equation*}
 F_y(s)=\sum_{P(n)\le y}\frac{f(n)}{n^s}
 =\begin{cases}
  \prod_{p\le y}(1-f(p)p^{-s})^{-1},
    &f \text{ Steinhaus},\\
  \prod_{p\le y}(1+f(p)p^{-s}),
    &f \text{ Rademacher},
 \end{cases}
\end{equation*}
denote a partial Dirichlet series and its corresponding Euler product.
Define
\begin{equation*}
 \mathcal E_y=\int_0^\infty|\Psi_f(z,y)|^2\,\frac{dz}{z^2},
\end{equation*}
and put $Z_y=\mathcal E_y/C_f(y)$, where $C_f(y)=\E|F_y(1/2+it)|^2$. Note that this expectation is independent of $t$, and
\begin{equation*}
 C_f(y)=
 \begin{cases}
  \prod_{p\le y}(1-1/p)^{-1},&f \text{ Steinhaus},\\
  \prod_{p\le y}(1+1/p),&f \text{ Rademacher}.
 \end{cases}
\end{equation*}
All empty products above are understood to be $1$.
For $y\geq 3$, \cref{lem:arithmetic}\textnormal{(i)} gives $C_f(y)\asymp\log y$.

Harper identifies normalized local mean squares of random Euler products with the total mass of a critical, approximately Gaussian multiplicative chaos. We use the corresponding whole-line integral selected naturally by Parseval's identity. 

\begin{lemma}\label[lemma]{lem:critical-energy}
For every $y\ge1$,
\begin{equation*}
  \mathcal E_y
  =\frac1{2\pi}\int_{\mathbb R}
    \Big|\frac{F_y(1/2+it)}{1/2+it}\Big|^2\,dt,
\end{equation*}
and $(Z_y)_y$ is a nonnegative mean-one martingale.
Moreover, for every prime $p$, there is an absolute constant $C$ such that
\begin{equation*}
  Z_p\le C Z_{p^-},
\end{equation*}
where $Z_{p^-}$ denotes the corresponding quantity formed using only primes $q<p$.
\end{lemma}

\begin{proof}
For fixed $y$, the Steinhaus Dirichlet series is absolutely convergent throughout $\Re s>0$, while the Rademacher series is a Dirichlet polynomial. Parseval's formula for Dirichlet series therefore gives the displayed identity (see, e.g., \cite[(5.26)]{MontgomeryVaughan}).

For every real $t$, the Euler factor contributed by the newly revealed prime $p$ has squared mean $1+1/p$ in the
Rademacher model and $(1-1/p)^{-1}$ in the Steinhaus model.  These are
independent of $t$ and equal the corresponding increment of $C_f$.
Since the integrand is nonnegative, we may pass the conditional expectation through the integral, thus proving
$\E[Z_p\mid \F_{p^-}]=Z_{p^-}$. Since $Z_y$ is held constant between primes, we deduce that it forms a martingale. We also see that $\E Z_p = \E Z_{p^-}$, which combined with the identity
\begin{equation*}
  Z_1=\frac1{2\pi}\int_{\mathbb R}\frac{dt}{1/4+t^2}=1
\end{equation*}
implies that it has mean one.

Finally, bound the revealed local factor pointwise in the spectral integral to get
\begin{equation*}
 \frac{Z_p}{Z_{p^-}}
 \le
 \begin{cases}
  (1-1/p)(1-p^{-1/2})^{-2}
  =\dfrac{1+p^{-1/2}}{1-p^{-1/2}},&f\text{ Steinhaus},\\[7pt]
  \dfrac{(1+p^{-1/2})^2}{1+1/p},&f\text{ Rademacher}.
 \end{cases}
\end{equation*}
Both expressions are uniformly bounded for $p\ge2$.
\end{proof}

We shall use one consequence of Harper's critical chaos analysis.  For
$2/3\le q<1$, write
\begin{equation*}
 D_q(y)=\min\Big\{1,\frac1{(1-q)\sqrt{\log_2y}}\Big\}.
\end{equation*}

\begin{lemma}\label[lemma]{lem:critical-moment}
Uniformly for large $y$ and $2/3\le q<1$,
\begin{equation*}
  \E Z_y^q\ll D_q(y)^q.
\end{equation*}
\end{lemma}

\begin{proof}
If $D_q(y)=1$, the assertion follows from concavity and $\E Z_y=1$.
Suppose henceforth that $D_q(y)<1$.  Thus we are in the range where the desired estimate is nontrivial. Let
$I_N=[N-1/2,N+1/2]$ and
\begin{equation*}
  H_N(y)=\frac1{\log y}
  \int_{I_N}|F_y(1/2+it)|^2\,dt.
\end{equation*}
In Harper's notation, $F_k$ is the Euler product over primes at most
$x^{e^{-(k+1)}}$. Thus, on taking $x=y^e$ and $k=0$ in his notation, one
has $F_0=F_y$, $\log x=e\log y$, and $\log_2x=\log_2y+1$.
Harper's Steinhaus upper bound argument, based on Key
Propositions~1 and~2 from \cite[Sections~4.1--4.3]{HarperLow}, then gives
\begin{equation*}
  \E H_0(y)^q\ll D_q(y)^q.
\end{equation*}
Replacing $f(p)$ with $f(p)p^{-iN}$ preserves its joint law and translates
the spectral variable.  The same estimate therefore holds uniformly for
every $N\in\mathbb Z$.

For the Rademacher model, put $\delta_y=(\log_2y)^{-1/2}$ and define
\begin{equation*}
 H_N^\circ(y)=\frac1{\log y}
 \int_{I_N\cap\{|t|>\delta_y\}}|F_y(1/2+it)|^2\,dt.
\end{equation*}
The corresponding shifted estimate follows from Key
Propositions~3 and~4 and the argument in \cite[Section~4.4]{HarperLow}.  With the same choices
$x=y^e$ and $k=0$ in Harper's notation, it gives, for
$|N|\le(\log_2y)^2$,
\begin{equation*}
  \E H_N^\circ(y)^q
  \ll D_q(y)^q(1+|N|)^{q/4}.
\end{equation*}
Harper's cutoff $(\log_2x)^{-1/2}$ is smaller than $\delta_y$, so our
truncated integral is bounded by his.  For large $y$, only $I_0$ meets
$[-\delta_y,\delta_y]$.  Since
$\E|F_y(1/2+it)|^2=C_f(y)\asymp\log y$, concavity shows that the omitted
part has $q$th moment
\begin{equation*}
 \ll \delta_y^q =(\log_2y)^{-q/2}\leq D_q(y)^q,
\end{equation*}
so the same bound holds with $H_N^\circ(y)$ replaced by $H_N(y)$.  For every
$N$, concavity also gives the trivial estimate $\E H_N(y)^q\ll1$.

On $I_N$, one has $|1/2+it|^{-2}\ll(1+N^2)^{-1}$.  The identity in
\cref{lem:critical-energy} and $C_f(y)\asymp\log y$ therefore give
\begin{equation}\label{eq:spectral-domination}
  Z_y\ll\sum_{N\in\mathbb Z}(1+N^2)^{-1}H_N(y).
\end{equation}
Since $2/3\leq q<1$, subadditivity of the $q$th power applied to
\eqref{eq:spectral-domination} gives in the Steinhaus model
\begin{equation*}
  \E Z_y^q
  \ll D_q(y)^q\sum_{N\in\mathbb Z}(1+N^2)^{-q}
  \ll D_q(y)^q.
\end{equation*}
In the Rademacher model, \eqref{eq:spectral-domination} and the preceding
moment estimates show that the terms with $|N|\le(\log_2y)^2$ contribute at most
\begin{equation*}
 D_q(y)^q\sum_{N\in\mathbb Z}(1+|N|)^{-2q+q/4}
 \ll D_q(y)^q.
\end{equation*}
The remaining terms contribute
\begin{equation*}
 \ll\sum_{|N|>(\log_2y)^2}|N|^{-2q}
 \ll(\log_2y)^{2-4q}.
\end{equation*}
As $q\ge2/3$ and $D_q(y)\gg(\log_2y)^{-1/2}$, this is
$O(D_q(y)^q)$.
\end{proof}

\section{The prime martingale and its quadratic clock}
\noindent We will apply \cref{lem:sparse-grid} with $A=1$, so fix the corresponding
$c_0\in(0,1)$ and test points $(x_i)$. 

We divide the prime range into layers of unit length in the natural Euler-product time $\log_2y$.
For an integer $\ell$ sufficiently
large, set
\begin{equation*}
 \Lambda=\Lambda_\ell=\lceil e^\ell\rceil,
 \qquad
 X_\ell=\exp(\exp(\Lambda_\ell)),
\end{equation*}
and consider the macroblock $X_\ell<x\le X_{\ell+1}$.  Put
\begin{equation*}
 \tau_0=\Lambda_\ell-2\ell,
 \qquad
 J=\Lambda_{\ell+1}-\Lambda_\ell+2\ell,
\end{equation*}
and, for $0\leq j\leq J$, set
\begin{equation*}
 \tau_j=\tau_0+j,
 \qquad
 y_j=\exp(\exp(\tau_j)).
\end{equation*}
Thus the prime layers are $(y_{j-1},y_j]$, $1\leq j\leq J$, and
\begin{equation*}
J\asymp\Lambda,\qquad \log y_j=e\log y_{j-1},\qquad y_J=X_{\ell+1}.
\end{equation*}
Throughout this macroblock, $\Lambda\asymp\log_2x$,
$\ell\asymp\log_3x$, and $\tau_j=\log_2y_j$.  The fixed constant $c_0$
affects only the implied constants below.
Moreover, the number of test points in this macroblock is at most
\begin{equation}\label{eq:grid-cardinality}
 (2\log X_{\ell+1})^{1/c_0}
 =\exp(O_{c_0}(\Lambda)).
\end{equation}

\subsection{Largest prime reduction}

For $x\le X_{\ell+1}$, define
\begin{equation*}
  \mathcal M_\ell(x)
  =\sum_{y_0<p\le y_J}f(p)\Psi_f'(x/p,p).
\end{equation*}
In the Steinhaus model, put
\begin{equation*}
  \mathcal R_\ell(x)
  =\sum_{y_0<p\le y_J}
    \sum_{\substack{k\ge2\\p^k\le x}}
    f(p)^k\Psi_f'(x/p^k,p),
\end{equation*}
and set $\mathcal R_\ell(x)=0$ in the Rademacher model.

\begin{lemma}\label[lemma]{lem:largest-prime}
For every $x\le X_{\ell+1}$,
\begin{equation*}
  M_f(x)=\Psi_f(x,y_0)+\mathcal M_\ell(x)+\mathcal R_\ell(x).
\end{equation*}
Moreover,
\begin{equation*}
  \E|\mathcal R_\ell(x)|^2\ll\frac{x}{y_0}.
\end{equation*}
\end{lemma}

\begin{proof}
Every integer outside $\Psi_f(x,y_0)$ has a unique representation $n=p^km$,
where $p>y_0$ is its largest prime factor and  $k\ge1$.  In the
Rademacher model the support is squarefree, so only $k=1$ occurs.

In the Steinhaus model, all integers represented in $\mathcal R_\ell(x)$ are
distinct.  Orthogonality therefore gives
\begin{equation*}
 \E|\mathcal R_\ell(x)|^2 = \sum_{y_0<p\le y_J}\sum_{k\geq 2}\#\{m\leq x/p^k : P(m)<p\}  \le\sum_{p>y_0}\sum_{k\ge2}\frac{x}{p^k} \ll x\sum_{p>y_0}\frac{1}{p^2}\ll \frac{x}{y_0}. \qedhere
\end{equation*}
\end{proof}

We now justify that we can discard all but the contributions of $\mathcal M_\ell(x)$.

\begin{lemma}\label[lemma]{lem:negligible}
Almost surely, for all sufficiently large $\ell$,
\begin{equation*}
 \max_{X_\ell<x_i\le X_{\ell+1}}
 \frac{|\Psi_f(x_i,y_0)|+|\mathcal R_\ell(x_i)|}{\sqrt{x_i}}
 \ll\Lambda^{-1}.
\end{equation*}
\end{lemma}

\begin{proof}
By \cref{lem:largest-prime} and Markov's inequality, the probability that the
repeated-prime term exceeds $\sqrt{x_i}\Lambda^{-1}$ at one test point is
$O(\Lambda^2/y_0)$.  Multiplication by the number of test points in
\eqref{eq:grid-cardinality} leaves a
summable series in $\ell$.

For the smooth term, orthogonality and
\cref{lem:arithmetic}\textnormal{(iii)} give
\begin{equation}\label{eq:smooth-second-moment}
\E|\Psi_f(x_i,y_0)|^2 \leq \#\{n\le x_i:P(n)\le y_0\}
  \ll x_i(\log y_0)^C
  \exp\big(-\frac{\log x_i}{2\log y_0}\big).
\end{equation}
Since $x_i>X_\ell$,
\begin{equation}\label{eq:smooth-separation}
  \frac{\log x_i}{\log y_0}\ge e^{2\ell}\asymp\Lambda^2.
\end{equation}
By \eqref{eq:smooth-second-moment} and \eqref{eq:smooth-separation}, Markov's inequality at threshold
$\sqrt{x_i}\Lambda^{-1}$ 
gives probability 
$\ll \exp(O(\Lambda)-c\Lambda^2)$
at each test point, for some absolute $c>0$.  Here
$(\log y_0)^C=\exp(C\log_2y_0)=\exp(O(\Lambda))$.  By
\eqref{eq:grid-cardinality}, the same estimate with different constants holds after a union bound. These probabilities are summable in $\ell$, so Borel--Cantelli proves the lemma.
\end{proof}

Reveal the primes in $(y_0,y_J]$ in increasing order.  The coefficient
$\Psi_f'(x/p,p)$ is $\F_{p^-}$-measurable, while $f(p)\Psi_f'(x/p,p)$ has conditional mean zero. Thus $\mathcal M_\ell(x)$ is a martingale with
predictable quadratic variation
\begin{equation*}
  V_\ell(x)=\sum_{y_0<p\le y_J}|\Psi_f'(x/p,p)|^2.
\end{equation*}
Terms with $p>x$ vanish. Note that by \cref{lem:subgaussian}, for every $x,t,v>0$,
\begin{equation} \label{eq:main-subgaussian}
  \Prob\big(|\mathcal M_\ell(x)|>t,\ V_\ell(x)\le v\big)
  \le2\exp(-t^2/(10v)).
\end{equation}

\subsection{Smoothing into consecutive windows}

For $1\le j\le J$, define energies
\begin{equation*}
  \mathcal U_j
  =\frac1{\log y_j}
   \int_0^\infty
   \sup_{y_{j-1}\le u\le y_j}|\Psi_f(z,u)|^2
   \,\frac{dz}{z^2}.
\end{equation*}
For fixed $z$, the function $u\mapsto \Psi_f(z,u)$ is a step function whose jumps occur only at primes $p\leq z$. The supremum is therefore a maximum over a finite set.
For a test point $x$ in the macroblock, let
\begin{equation*}
 \mathcal N_\ell(x)
 =\Big\{1\le j\le J:
     y_{j-1}<x,\,
     \frac{\log x}{\log y_{j-1}}\le\Lambda^2
   \Big\}.
\end{equation*}
This is an interval of consecutive indices.  Indeed, membership is equivalent
to
\begin{equation*}
 \log_2x-2\log\Lambda\le\tau_{j-1}<\log_2x.
\end{equation*}
Since $\log\Lambda=\ell+O(1)$ and the $\tau_j$ are spaced by one, there is an absolute $C'>0$ such that
\begin{equation*}
  |\mathcal N_\ell(x)|\le C'\ell
\end{equation*}
for all large $\ell$.  Let $K_\ell=\lceil C'\ell\rceil$.

\begin{lemma}\label[lemma]{lem:smoothing}
There are events $\mathcal A_\ell$ such that
\begin{equation*}
  \sum_{\ell\ge1}\Prob(\mathcal A_\ell)<\infty
\end{equation*}
and, outside $\mathcal A_\ell$,
\begin{equation*}
  \max_{X_\ell<x_i\le X_{\ell+1}}
  \frac{V_\ell(x_i)}{x_i}
  \ll
  \max_{1\leq k\leq J}\sum_{j=k}^{\min(k+K_\ell-1,J)}\mathcal U_j
  +\Lambda^{-1}.
\end{equation*}
\end{lemma}

\begin{proof}
For a test point $x=x_i$, put
$Q=(\log x)^{2/c_0}$, and write
\begin{align*}
  L_\ell(x)
  &=\sum_{y_0<p\le y_J}\frac Qp
    \int_p^{p(1+1/Q)}|\Psi_f'(x/t,p)|^2\,dt,\\
  W_\ell(x)
  &=\sum_{y_0<p\le y_J}\frac Qp
    \int_p^{p(1+1/Q)}
      |\Psi'_f(x/p,p)-\Psi'_f(x/t,p)|^2\,dt.
\end{align*}
Using that the averaging measure $Q/p\,dt$ has total mass one on this interval and the inequality $|a+b|^2 \leq 2|a|^2+2|b|^2$, we see that
\begin{equation}\label{eq:clock-split}
  V_\ell(x)\le2L_\ell(x)+2W_\ell(x).
\end{equation}

The change of variables $z=x/t$ and Tonelli give
\begin{equation*}
 \frac{L_\ell(x)}x
 =\int_0^\infty
  Q\!\sum_{\substack{y_0<p\le y_J\\
       x/[z(1+1/Q)]<p\le x/z}}
  \frac{|\Psi_f'(z,p)|^2}{p}\,\frac{dz}{z^2}.
\end{equation*}
If $p\in(y_{j-1},y_j]$, let $u_p$ be the largest prime in
$(y_{j-1},p)$, and put $u_p=y_{j-1}$ when this interval contains no such prime.
Then
\begin{equation*}
  \Psi_f'(z,p)=\Psi_f(z,u_p),
  \qquad
  |\Psi_f'(z,p)|^2
  \le\sup_{y_{j-1}\le u\le y_j}|\Psi_f(z,u)|^2.
\end{equation*}
Partition $L_\ell(x)=L_{\ell,\mathrm{near}}(x)
+L_{\ell,\mathrm{far}}(x)$ according to whether the layer containing $p$ does or
does not belong to $\mathcal N_\ell(x)$. Thus $L_{\ell,\mathrm{near}}$ is the contribution of the primes $p\in(y_{j-1},y_j]$ with $j\in \mathcal N_\ell(x)$, and $L_{\ell,\mathrm{far}}$ is the contribution of the remaining layers. A nonzero term in the far part necessarily comes from a layer satisfying
\begin{equation} \label{eq:farL}
y_{j-1}<x,\qquad \frac{\log x}{\log y_{j-1}}>\Lambda^2.
\end{equation}

Write $b=x/z$.  Whenever the prime sum is nonempty, $b\ge y_0$, so
$Q\le b^{1/2}$ for all sufficiently large $\ell$. Indeed,
$\log Q=O(\Lambda)$, whereas
$\log(y_0^{1/2})=\frac{1}{2}\exp(\Lambda-2\ell)$.  
Applying
\cref{lem:arithmetic}\textnormal{(ii)} gives
\begin{equation*}
  Q\sum_{b/(1+1/Q)<p\le b}\frac1p
  \ll\frac1{\log b}.
\end{equation*}
Decompose the prime window $(b/(1+1/Q),b]$ according to its intersection with the layers $(y_{j-1},y_j]$. 
The same upper bound holds for each layer, and if the $j$-th one is nonempty, then
$\log b\asymp\log y_j$. Thus the reciprocal mass of each layer piece is
$O(1/(Q\log y_j))$.  Bounding $|\Psi_f'(z,p)|^2$ by the supremum over the corresponding
layer and extending the resulting $z$-integral to all $z>0$, we obtain
\begin{equation}\label{eq:near-clock}
 \frac{L_{\ell,\mathrm{near}}(x)}x
 \ll
 \sum_{j\in\mathcal N_\ell(x)}
 \frac1{\log y_j}
 \int_0^\infty
 \sup_{y_{j-1}\le u\le y_j}|\Psi_f(z,u)|^2
 \,\frac{dz}{z^2}
 =\sum_{j\in\mathcal N_\ell(x)}\mathcal U_j.
\end{equation}

Consider a nonzero contribution to $L_{\ell,\mathrm{far}}(x)$ from a prime $p\in (y_{j-1},y_j]$. Then \eqref{eq:farL} holds, and
since
$\log p\le\log y_j=e\log y_{j-1}$ and
$t\le p(1+1/Q)$, we get
\begin{equation*}
  \frac{\log(x/t)}{\log p}
  \ge\frac{\Lambda^2}{e}-1
     -\frac{\log(1+1/Q)}{\log p}
  \ge c\Lambda^2
\end{equation*}
for some $c>0$.
Orthogonality and \cref{lem:arithmetic}\textnormal{(iii)} therefore give
\begin{equation*}
  \E|\Psi_f'(x/t,p)|^2 \ll \frac{x}{t}(\log p)^C\exp\big(-\frac{\log(x/t)}{2\log p}\big)
  \ll\frac xt\exp(-c\Lambda^2),
\end{equation*}
where the fixed power of $\log p$ has been absorbed and $c>0$ is a (possibly different) constant.  Since
\begin{equation*}
 \frac Qp\int_p^{p(1+1/Q)}\frac xt\,dt
 =\frac{xQ}{p}\log(1+1/Q)
 \ll\frac xp,
\end{equation*}
it follows that
\begin{equation}\label{eq:far-clock-mean}
  \E L_{\ell,\mathrm{far}}(x) \ll \exp(-c\Lambda^2) \sum_{p\leq y_J}\frac{x}{p}
  \ll x\Lambda\exp(-c\Lambda^2),
\end{equation}
where we used \cref{lem:arithmetic}\textnormal{(i)}.
After taking the union over the test points using \eqref{eq:grid-cardinality}, the
probability that $L_{\ell,\mathrm{far}}(x_i)/x_i>\Lambda^{-1}$ somewhere in
the macroblock is summable in $\ell$, by Markov's inequality.

It remains to control the error caused by replacing the endpoint $x/p$ with $x/t$.  By orthogonality,
\begin{equation*}
  \E W_\ell(x)
  \leq \sum_{y_0<p\le y_J}\frac Qp
    \int_p^{p(1+1/Q)}
      \sum_{\substack{x/t<m\leq x/p \\ P(m)<p}}1\,dt.
\end{equation*}
After Tonelli, fix a pair $(p,m)$ counted on the right.  Its contribution to
the $t$-integral, including the factor $Q/p$, is at most one.  Its existence
forces
\begin{equation*}
  \frac{x}{1+1/Q}<pm\le x,
  \qquad P(m)<p.
\end{equation*}
Thus $p=P(pm)$ and occurs to exponent one, so each integer $pm$ is
represented at most once. Therefore, the analytic error is converted into an elementary count of integers lying in a short interval immediately below $x$, and in fact
\begin{equation}\label{eq:boundary-clock-mean}
  \E W_\ell(x)\leq \frac{x}{Q+1}+1\ll\frac{x}{Q}.
\end{equation}
Since $\log x_i\asymp i^{c_0}$, Markov's inequality gives, for some $\ell_0$ sufficiently large,
\begin{equation*}
 \sum_{\ell\geq \ell_0}\sum_{X_\ell <x_i\leq X_{\ell+1}}
  \Prob\Big(
    \frac{W_\ell(x_i)}{x_i}>(\log_2x_i)^{-1}
  \Big)
 \ll\sum_i
    \frac{\log i}{i^2}
  <\infty.
\end{equation*}

Combining \eqref{eq:clock-split} and \eqref{eq:near-clock} with the preceding
exceptional event estimates obtained from \eqref{eq:far-clock-mean} and
\eqref{eq:boundary-clock-mean}, and using $\log_2x_i \asymp \Lambda$, we get
\begin{equation*}
  \frac{V_\ell(x_i)}{x_i}
  \ll\sum_{j\in\mathcal N_\ell(x_i)}\mathcal U_j+\Lambda^{-1}.
\end{equation*}
If $\mathcal N_\ell(x_i)$ is nonempty and $k$ is its first index, then the fact that it contains consecutive indices and the bound $|\mathcal N_\ell(x_i)|\leq K_\ell$ give
\begin{equation*}
\mathcal N_\ell(x_i) \subseteq \{k,\ldots,\min(k+K_\ell-1,J)\}.
\end{equation*}
Since the $\mathcal U_j$ are nonnegative, taking the maximum over $x_i$ proves the lemma.
\end{proof}

\section{A stopped consecutive-window estimate}

The next lemma is the probabilistic core of the argument.  The proof establishes a conditional exponential moment bound for each stopped layer energy, and iterating this bound controls sums over consecutive layers without requiring independence.

\begin{lemma}\label[lemma]{lem:windows}
Let $2\le y_0<\cdots<y_J$, define $\mathcal U_j$ as above, and let
$\mathcal I$ be a nonempty deterministic family of intervals in
$\{1,\ldots,J\}$, each containing at most $L\geq 1$ indices.
There is an absolute
constant $C$ such that, for $0<q<1$, $a>0$, and $s\ge0$,
\begin{equation*}
 \Prob\Big(
   \max_{I\in\mathcal I}\sum_{j\in I}\mathcal U_j
   >Ca\big(L+\log |\mathcal I|+s\big)
 \Big)
 \le a^{-q}\E Z_{y_0}^q+e^{-s}.
\end{equation*}
\end{lemma}

\begin{proof}
Put $\mathcal B_a=\{Z_{y_0}\le a\}$.  On $\mathcal B_a$, let $\sigma_a$ be
the first prime $p\in(y_0,y_J]$ for which $Z_p>a$, taking $\sigma_a=y_J$ if
there is no such prime.  On $\mathcal B_a^c$, put $\sigma_a=y_0$.  This is a
bounded stopping time for the prime filtration.  For $u\ge y_0$, define
\begin{equation*}
  \widetilde\Psi_a(z,u)
  =\1_{\mathcal B_a}\Psi_f(z,u\wedge\sigma_a)
\end{equation*}
and
\begin{equation*}
 \mathcal U_j^{(a)}
 =\frac1{\log y_j}\int_0^\infty
   \sup_{y_{j-1}\le u\le y_j}
   |\widetilde\Psi_a(z,u)|^2\,\frac{dz}{z^2}.
\end{equation*}

By \cref{lem:critical-energy}, the stopped energy
can exceed $a$ by only an absolute factor.  Consequently, for every $j$ we get
\begin{equation}\label{eq:stopped-terminal-energy}
 \frac1{\log y_j}\int_0^\infty
 |\widetilde\Psi_a(z,y_j)|^2\,\frac{dz}{z^2}
 =\1_{\mathcal B_a}
 \frac{C_f(y_j\wedge\sigma_a)}{\log y_j}
 Z_{y_j\wedge\sigma_a}
 \ll a
\end{equation}
pathwise, where we used $C_f(y)\asymp \log y$.

For fixed $z$, the process $u\mapsto\Psi_f(z,u)$ is a martingale in the
prime cutoff. Indeed, in the Steinhaus model, the increment at a newly revealed prime $p$
is a sum of terms containing $f(p)^k$ with $k\ge1$, and
$\E f(p)^k=0$, while in the Rademacher model only the term $k=1$ occurs, which also satisfies $\E f(p)=0$. Since
$\mathcal B_a\in\F_{y_0}$, the stopped process
$\widetilde\Psi_a(z,u)$ is a martingale on the filtration beginning at
$\F_{y_0}$.  In particular, optional sampling in the finite prime filtration
gives, for
$y_{j-1}\le u\le y_j$,
\begin{equation*}
  \widetilde\Psi_a(z,u)
  =\E[\widetilde\Psi_a(z,y_j)\mid\F_u].
\end{equation*}

Let
$A\in\F_{y_{j-1}}$. Since $A$ is measurable at the beginning of the $j$-th layer,
\begin{equation*}
 \1_A\widetilde\Psi_a(z,u)
 =\E[\1_A\widetilde\Psi_a(z,y_j)\mid\F_u].
\end{equation*}
For fixed $u$, the function $z\mapsto \widetilde\Psi_a(z,u)$ is constant on each interval $[n,n+1)$. For
$n\ge1$, use the coordinate weight
\begin{equation*}
  w_n=\frac1{\log y_j}\frac1{n(n+1)}.
\end{equation*}
Indeed, $\widetilde\Psi_a(z,u)=\widetilde\Psi_a(n,u)$ for
$n\le z<n+1$, while
\begin{equation*}
 \int_n^{n+1}\frac{dz}{z^2}=\frac1{n(n+1)}.
\end{equation*}
Since the integrand vanishes for $z<1$, the integral defining
$\mathcal U_j^{(a)}$ is exactly
\begin{equation*}
 \mathcal U_j^{(a)}
 =\sum_{n\ge1}w_n
   \sup_{y_{j-1}\le u\le y_j}|\widetilde\Psi_a(n,u)|^2.
\end{equation*}
Fix an integer $m\geq2$ and truncate the coordinates at $R$.  Let
$p_1<\cdots<p_N$ be the primes in $(y_{j-1},y_j]$, and apply
\cref{lem:vector-doob} with $r=m$, filtration $\mathcal G_0=\F_{y_{j-1}}$ and $\mathcal G_k=\F_{p_k}$,
terminal variables $X_n=\1_A\widetilde\Psi_a(n,y_j)$ for $1\leq n\leq R$,
and weights $w_n$. Since
\begin{equation*}
 \E[X_n\mid\mathcal G_k]
 =\1_A\widetilde\Psi_a(n,p_k),
\end{equation*}
and the cutoff process is constant between prime jumps, this gives
\begin{equation*}
 \E\Big[
  \1_A\Big(
   \sum_{n=1}^Rw_n
   \sup_{y_{j-1}\leq u\leq y_j}
   |\widetilde\Psi_a(n,u)|^2
  \Big)^m
 \Big]
 \leq
 (Cm)^m
 \E\Big[
  \1_A\Big(
   \sum_{n=1}^Rw_n
   |\widetilde\Psi_a(n,y_j)|^2
  \Big)^m
 \Big].
\end{equation*}
The expression on the right is at most
$(Cma)^m\Prob(A)$ by \eqref{eq:stopped-terminal-energy}.
Letting $R\to\infty$ and using monotone convergence yields
\begin{equation*}
 \E\big[\1_A(\mathcal U_j^{(a)})^m\big]
 \leq(Cma)^m\Prob(A).
\end{equation*}
Since this holds for every $A\in\F_{y_{j-1}}$,
\begin{equation}\label{eq:conditional-layer-moments}
 \E\big[(\mathcal U_j^{(a)})^m\mid\F_{y_{j-1}}\big]
 \le(Cma)^m
\end{equation}
almost surely.  The same estimate for $m=1$ follows from conditional
Cauchy--Schwarz and the case $m=2$.

Using \eqref{eq:conditional-layer-moments} and $m!\ge(m/e)^m$ in the exponential series, choose a sufficiently small $\theta>0$ to obtain
\begin{equation*}
 \E\big[
   \exp(\theta\mathcal U_j^{(a)}/a)
   \mid\F_{y_{j-1}}
 \big] \leq
 1+\sum_{m\geq1}\frac{(Cm\theta)^m}{m!} \leq 1+\sum_{m\geq 1}(eC\theta)^m\leq 2.
\end{equation*}
For $u\le y_j$, the value $u\wedge\sigma_a$ is determined by $\F_u$,
so $\mathcal U_j^{(a)}$ is $\F_{y_j}$-measurable.  Iterating
this inequality backwards through a fixed interval $I$ gives
\begin{equation}\label{eq:window-exponential}
 \E\exp\Big(\frac\theta a
   \sum_{j\in I}\mathcal U_j^{(a)}\Big)
 \le2^{|I|}.
\end{equation}
For a fixed $I\in\mathcal I$, \eqref{eq:window-exponential} and the exponential form of Markov's inequality give
\begin{equation*}
 \Prob\Big(
  \sum_{j\in I}\mathcal U_j^{(a)}>T
 \Big)
 \leq
 \exp\Big(L\log2-\frac{\theta T}{a}\Big).
\end{equation*}
Taking
$
 \displaystyle T=\frac a\theta\big(L\log2+\log|\mathcal I|+s\big)
$
makes the right-hand side at most
$e^{-s}/|\mathcal I|$.  A union bound over $I\in\mathcal I$ therefore
shows that, outside an event of probability at most $e^{-s}$,
\begin{equation*}
 \max_{I\in\mathcal I}\sum_{j\in I}\mathcal U_j^{(a)}
 \leq
 Ca\big(L+\log |\mathcal I|+s\big).
\end{equation*}

Finally, $(Z_u^q)_u$ is a nonnegative supermartingale.  Doob's maximal
inequality gives
\begin{equation*}
 \Prob\Big(\sup_{y_0\le u\le y_J}Z_u>a\Big)
 \le a^{-q}\E Z_{y_0}^q.
\end{equation*}
On the complementary event, $\sigma_a=y_J$ and
$\mathcal U_j^{(a)}=\mathcal U_j$ for every $j$.  Combining the preceding
two estimates proves the lemma.
\end{proof}

\section{Proof of the main theorem}
Use the same notation as above.
After changing $G$ on a bounded interval, replace it by
\begin{equation*}
  \min\{G(t),t^{3/2}\}.
\end{equation*}
This function is still eventually nondecreasing, gives a stronger conclusion, and preserves the hypothesis because the new summand is $(\log\ell)/\ell^2$ wherever the function is changed.  We may therefore suppose that $G(t)\le t^{3/2}$ for all large $t$.  Put
\begin{equation*}
  B_\ell=\frac{G(\ell)^2}{\ell}.
\end{equation*}
Then the hypothesis becomes
\begin{equation}\label{eq:B-summability}
 \sum_{\ell\geq3}\frac{\log\ell}{B_\ell}<\infty,
\end{equation}
while the preceding reduction gives $B_\ell\leq\ell^2$.  In particular,
$\log\ell/B_\ell\to0$.
Set
\begin{equation*}
  a_\ell=\frac{B_\ell}{\sqrt\Lambda},
  \qquad
  q_\ell=1-\frac1{\log\ell}.
\end{equation*}
For all sufficiently large $\ell$, one has $2/3\le q_\ell<1$.  Since
$\log_2y_0=\Lambda-2\ell\asymp\Lambda$ and
\begin{equation*}
  \frac{\log\ell}{\sqrt{\log_2y_0}}
  \asymp\frac{\log\ell}{\sqrt\Lambda}=o(1),
\end{equation*}
\cref{lem:critical-moment} gives
\begin{equation*}
  \E Z_{y_0}^{q_\ell}
  \ll\Big(\frac{\log\ell}{\sqrt\Lambda}\Big)^{q_\ell}.
\end{equation*}
Consequently,
\begin{equation}\label{eq:stopping-term}
 a_\ell^{-q_\ell}\E Z_{y_0}^{q_\ell}
 \ll\Big(\frac{\log\ell}{B_\ell}\Big)^{q_\ell}
 \ll\frac{\log\ell}{B_\ell}.
\end{equation}
The last inequality holds because, if
$u=\log\ell/B_\ell$, then $u<1$ eventually, yet 
\begin{equation*}
  u^{q_\ell}
  =u\exp\Big(\frac{\log(1/u)}{\log\ell}\Big)
  \ll u,
\end{equation*}
using $1/u=B_\ell/\log\ell \le\ell^2$.

Apply \cref{lem:windows} with
\begin{equation*}
 a=a_\ell,\qquad q=q_\ell,\qquad
 L=K_\ell,
 \qquad s=2\log\ell,
\end{equation*}
and with $\mathcal I$ equal to the family of intervals
\begin{equation*}
 \{k,\ldots,\min(k+K_\ell-1,J)\},
 \qquad 1\leq k\leq J.
\end{equation*}
Since $K_\ell\ll\ell$, $|\mathcal I|\leq J$, and $\log J\ll\ell$,
\eqref{eq:stopping-term} gives
\begin{equation*}
 \Prob\Big(
  \max_{1\leq k\leq J}
  \sum_{j=k}^{\min(k+K_\ell-1,J)}\mathcal U_j
  >
  C\frac{\ell B_\ell}{\sqrt\Lambda}
 \Big)
 \ll
 \frac{\log\ell}{B_\ell}+\ell^{-2},
\end{equation*}
for a sufficiently large absolute constant $C$.
The right-hand side is summable by \eqref{eq:B-summability}. Hence
Borel--Cantelli shows that, almost surely for all sufficiently large
$\ell$,
\begin{equation*}
 \max_{1\leq k\leq J}
 \sum_{j=k}^{\min(k+K_\ell-1,J)}\mathcal U_j
 \ll
 \frac{\ell B_\ell}{\sqrt\Lambda}.
\end{equation*}
Together with \cref{lem:smoothing}, this gives
\begin{equation}\label{eq:clock-bound}
  \max_{X_\ell<x_i\le X_{\ell+1}}
  \frac{V_\ell(x_i)}{x_i}
  \ll\frac{\ell B_\ell}{\sqrt\Lambda}.
\end{equation}
The term $\Lambda^{-1}$ has been absorbed.  Indeed,
\eqref{eq:B-summability} gives $B_\ell/\log\ell\to\infty$, and hence
$\Lambda^{-1}=o(\ell B_\ell/\sqrt\Lambda)$.

For a test point in the macroblock, put
\begin{equation*}
  t_i=D\sqrt{x_i}\Lambda^{1/4}
       \sqrt{\ell B_\ell},
  \qquad
  v_i=Cx_i\frac{\ell B_\ell}{\sqrt\Lambda},
\end{equation*}
where $C$ is large enough for \eqref{eq:clock-bound} and $D$ is a sufficiently large constant.  
Then \eqref{eq:main-subgaussian} and \eqref{eq:grid-cardinality}, together with a union
bound, give
\begin{align*}
 \Prob\big(\exists\, i:\ |\mathcal M_\ell(x_i)|>t_i,
                     \ V_\ell(x_i)\le v_i\big)
 &\le
 2\sum_{X_\ell<x_i\leq X_{\ell+1}}\exp(-t_i^2/(10v_i)) \\
 &\le2\exp(O_{c_0}(\Lambda)-D^2\Lambda/(10C)) \\
 &\ll e^{-c'\Lambda}
\end{align*}
for some $c'>0$.
These probabilities are summable in $\ell$. The quadratic variation bound fails only on another summable sequence of events, so the probabilities $\Prob(\exists\, i: |\mathcal M_\ell(x_i)|>t_i)$ are also summable.
Hence, another application of Borel--Cantelli and
\cref{lem:largest-prime,lem:negligible} give
\begin{equation}\label{eq:test-point-bound}
  |M_f(x_i)|
  \ll_{G}\sqrt{x_i}\Lambda^{1/4}
       \sqrt{\ell B_\ell}
  =\sqrt{x_i}\Lambda^{1/4}G(\ell)
\end{equation}
at every sufficiently large test point.

Now we extend this estimate from test points to all real $x$.
Let $x_{i-1}<x\leq x_i$, and suppose that
$X_\ell<x\leq X_{\ell+1}$.  By \cref{lem:sparse-grid},
\begin{equation*}
 |M_f(x)-M_f(x_{i-1})|
 \ll \frac{\sqrt{x_i}}{\log x_i}.
\end{equation*}
Let $r$ be the macroblock index containing $x_{i-1}$.  Since
$x_i/x_{i-1}\to1$, whereas $X_\ell/X_{\ell-1}\to\infty$, one has
$r\in\{\ell-1,\ell\}$ for all sufficiently large $i$.  Consequently,
\begin{equation*}
 \Lambda_r\asymp\log_2x,
 \qquad
 G(r)\leq G(\log_3x),
\end{equation*}
by the eventual monotonicity of $G$.  The test-point bound
\eqref{eq:test-point-bound} therefore gives
\begin{equation*}
 |M_f(x_{i-1})|
 \ll_{G}
 \sqrt{x}(\log_2x)^{1/4}G(\log_3x).
\end{equation*}
Finally, $x_i/x\to1$, and $G$ is eventually bounded below by a positive
constant, so
\begin{equation*}
 \frac{\sqrt{x_i}}{\log x_i}
 =
 o\big(
 \sqrt{x}(\log_2x)^{1/4}G(\log_3x)
 \big).
\end{equation*}
This proves the theorem.

\subsection*{Acknowledgments} 
The author was supported by a Vanier Canada Graduate Scholarship during the preparation of this work and received travel support from NSERC through CRTAS-RTS for his visit to Tsinghua University. He thanks Wenbo Li for his hospitality during this visit. He also thanks Mo Dick Wong for drawing his attention to the question of a law of the iterated logarithm for random multiplicative functions during a workshop at the Fields Institute, Max Xu for helpful discussions, and Adam Harper for pointing out another connection between sums of random multiplicative functions and Fourier coefficients of multiplicative chaos.

\bibliographystyle{amsalpha}
\bibliography{reference}

\end{document}